\documentclass[]{amsart}
\usepackage{amsfonts,latexsym,rawfonts,amsmath,amssymb, amsthm}
\usepackage{latexsym,lscape,rawfonts}
\usepackage{CJK}
\usepackage{mathrsfs, enumitem}
\usepackage[title]{appendix}
\usepackage{galois}
\usepackage[all,cmtip]{xy}
\usepackage{extarrows,color}
\usepackage{times}

\newenvironment{thmbis}[1]
{\renewcommand{\thethm}{\ref{#1}$ $}%
	\addtocounter{thm}{-1}%
	\begin{thm}}
	{\end{thm}}

\newtheorem{thm}{Theorem}[section]

\newtheorem*{fixed point criterion}{Fixed point criterion}
\newtheorem{cor}[thm]{Corollary}
\newtheorem{lem}[thm]{Lemma}
\newtheorem{prop}[thm]{Proposition}

\theoremstyle{definition}
\newtheorem{defn}[thm]{Definition}
\newtheorem{exam}[thm]{Example}
\newtheorem{ques}[thm]{Question}
\theoremstyle{remark}

\numberwithin{equation}{section}

\newcommand{\Z}{\mathbb Z}

\newcommand{\fix}{\mathrm{Fix}}

\newcommand{\ind}{\mathrm{ind}}

\newcommand{\aut}{\mathrm{Aut}}
\newcommand{\out}{\mathrm{Out}}
\newcommand{\Out}{\mathrm{Out}}

\newcommand{\inn}{\mathrm{Inn}}

\newcommand{\id}{\mathrm{id}}

\newcommand{\Lef}{\mathrm{L}}
\newcommand{\Min}{\mathrm{M}}
\newcommand{\Nie}{\mathrm{N}}

\newcommand{\F}{\mathbf{F}}      

\def\a{\alpha}

\begin{document}
	
	\title{On Jiang's Bounded Index Property for products of nilmanifolds}

	\author{Peng Wang}
	\address{School of Mathematics and Statistics, Xi'an Jiaotong University, Xi'an 710049, CHINA}
	\email{wp580829@stu.xjtu.edu.cn}
	
	
	\author{Qiang Zhang}
	\address{School of Mathematics and Statistics, Xi'an Jiaotong University, Xi'an 710049, CHINA}
	\email{zhangq.math@mail.xjtu.edu.cn}
	
	\author{Xuezhi Zhao}
	\address{School of Mathematical Sciences, Capital Normal University, Beijing 100048, CHINA}
	\email{zhaoxve@mail.cnu.edu.cn}

	\thanks{The authors are partially supported by NSFC (No. 12271385, 12471066 and 12331003), the Shaanxi Fundamental Science Research Project for Mathematics and Physics (No. 23JSY027), and the Fundamental Research Funds for the Central Universities.}
	
	
	
	\date{\today}
	\begin{abstract}
		In 2023, Zhang and Zhao presented the first examples of aspherical manifolds lacking the Bounded Index Property (BIP) for fixed points. This answered a question posed by Jiang in 1998 in the negative. In this paper, we first extend the notion of BIP to that of iterates of selfmaps (BIP$_k$), and then demonstrate BIP$_k$ for certain products $M\times N$ of nilmanifolds. Finally, we give characterizations for the Lefschetz number, the Nielsen number, and the minimal number of fixed points of self-homotopy equivalences of $M\times N$.
	\end{abstract}
	\maketitle
	
	
	\section{Introduction}
	In this paper, unless otherwise specified, all maps considered are continuous, and all spaces are connected, triangulable, namely, they are homeomorphic to polyhedra. For a group $G$, let $\aut(G)$ (resp. $\out(G)$, $\inn(G)$) denote the automorphism (resp. outer automorphism, inner automorphism) group of $G$.
	
	For a selfmap $f$ of a space $X$, Nielsen fixed point theory (for details, see Jiang's monograph \cite{J83}) is concerned with
	the properties of the fixed point set
	$$\fix f:=\{x\in X\mid f(x)=x\}$$
	that splits into a disjoint union of \textit{fixed point classes}. For each fixed point class $\F$ of $f$, a homotopy invariant \textit{index} $\ind(f,\F)\in \Z$ is well-defined.
	The famous Lefschetz-Hopf theorem says that the sum of the indices of the fixed points of $f$ is equal to the \textit{Lefschetz number} $\Lef(f)$, which is defined as
	$$
	\Lef(f):=\sum_q(-1)^q\mathrm{Trace} (f_*: H_q(X;\mathbb{Q})\to H_q(X;\mathbb{Q})).
	$$
	
	A fixed point class $\F$ is \textit{essential} if $\ind(f, \F) \neq 0$. The number of essential fixed point classes of $f$ is called the \textit{Nielsen number} of $f$, denoted by $\Nie(f)$, which is a lower bound for the minimum number of fixed points in the homotopy class of $f$, denoted by
	$$\Min(f):=\min \{\#\fix g \mid g \simeq f\}.$$
	Recall that a well-known result of Jiang \cite{J80}\cite[p.22, Theorem 6.3]{J83} says: if $X$ is a compact connected polyhedron without local separating points and $X$ is not a surface (closed or with boundary) of negative Euler characteristic, then $\Min(f) = \Nie(f)$ for any map $f:X\to X$.

	A compact polyhedron $X$ is said to have the
	\emph{Bounded Index Property (BIP)} (resp. \emph{Bounded Index Property for
		Homeomorphisms (BIPH)}, \emph{Bounded Index Property for Homotopy
		Equivalences (BIPHE)}) if there is an integer $\mathcal{B}>0$ such that for
	any map (resp. homeomorphism, homotopy equivalence) $f:X\rightarrow X$ and
	any fixed point class $\mathbf{F}$ of $f$, the index $|\mathrm{ind}(f,%
	\mathbf{F})|\leq \mathcal{B}$. Clearly, if $X$ has BIP, then $X$ has BIPHE
	and hence has BIPH. For an aspherical closed manifold $M$, if the well-known
	Borel conjecture (any homotopy equivalence $f:M\rightarrow M$ is homotopic
	to a homeomorphism $g:M\rightarrow M$) is true, then $M$ has BIPHE if and
	only if it has BIPH.

	In 1998, Jiang \cite{J98} posed the following question.
	
	\begin{ques}\cite[Question 3]{J98}.\label{Jiang's question}
		Does every compact aspherical polyhedron $X$ (i.e. $\pi_i(X)=0$ for all $i>1$) have BIP or BIPH?
	\end{ques}
	
	In the past more than thirty years, many types of aspherical polyhedra had been showed to support Question \ref{Jiang's question} with a positive answer: infra-solvmanifolds have BIP \cite{Mc92}; graphs and hyperbolic surfaces have BIP \cite{J98, JWZ11, K97, K00, WZ25};  geometric 3-manifolds have BIPH \cite{JW}; orientable Seifert 3-manifolds with hyperbolic orbifolds have BIPH \cite{Z12}; products of hyperbolic surfaces have BIPH \cite{ZZ19};  products of negatively curved Riemannian manifolds have BIPHE \cite{Z14, Z15, ZY20}; and  finite wedges of compact surfaces each having non-positive Euler characteristic have BIP \cite{GK25}. But very recently in \cite{LWZ25, ZZ23}, the first counterexamples were given.
	
	\begin{thm}[\cite{LWZ25, ZZ23}]\label{ZZ cout-examp.} Let $\Sigma_g$ be a closed orientable surface of genus $g$, and let $S^1$ be the circle and $T^n$ the $n$-dimension torus. Then
		\begin{enumerate}
			\item $\Sigma_g\times S^1 (g\geq 2)$ has BIPH, but does not have BIP;
			\item  $\Sigma_g\times T^n (g, n\geq 2)$ does not have BIPH, and hence does not have BIP.
		\end{enumerate}
	\end{thm}
	
	In this paper, we extend the notions BIP, BIPH and BIPHE to iterates of $f:X\to X$.
	
	\begin{defn}
		A polyhedron $X$ is said to have BIP$_k$ (resp. BIPH$_k$, BIPHE$_k$) for an integer $ k>0$  if there exists an integer $\mathcal{B}$ such that for every map (resp. homeomorphism, homotopy equivalence) $f: X\to X$, we have
		$$|\ind(f^k,\F)|\leq \mathcal{B}$$
		for every fixed point class $\F$ of $f^k$.
	\end{defn}
	
	Note that BIP (resp. BIPH, BIPHE) $\Longleftrightarrow$ BIP$_1$ (resp. BIPH$_1$, BIPHE$_1$) $\Longrightarrow$ BIP$_k$ (resp. BIPH$_k$, BIPHE$_k$) for all integer $k>1$. In fact, there are aspherical manifolds lacking BIPH but having BIPH$_k$ for some integer $k>0$, see Example \ref{eg. closed 3-mfd}.
	
	The first result of this paper is the following.
	
	\begin{thm} \label{main thm 0}
		Let $\Sigma_g$ be a closed orientable surface of genus $g$, and let $S^1$ be the circle and $T^n$ the $n$-dimension torus. Then for any integer $k>0$,
		\begin{enumerate}
			\item $\Sigma_g\times S^1 (g\geq 2)$ has BIPH$_k$, but does not have BIP$_k$;
			\item  $\Sigma_g\times T^n (g, n\geq 2)$ does not have BIPH$_k$, and hence does not have BIP$_k$.
		\end{enumerate}
	\end{thm}
	
	Moreover, note that the surface $\Sigma_g$ is of even dimension and the torus $T^n$ is a nilmanifold. Recall that, by work of Mal$^{\prime}$cev \cite{Mal49}, a closed \textit{nilmanifold} $N$ can be identified with the quotient space $\widehat{N}/G$, where $\widehat{N}$ is a simply connected nilpotent Lie group and $G \subset \widehat N$ is a cocompact lattice (which we identify with $\pi_1(N)$).  In contrast to Theorem \ref{main thm 0}, we have:
	
	\begin{thm}\label{main thm 1}
		Let $M$ be a closed negatively curved Riemannian manifold of odd dimension, and let $N$ be a closed nilmanifold. Then for any integer $k>0$ divisible by the order $|\out(\pi_1(M))|$, the product $M\times N$ has BIPHE$_k$ (and hence has BIPH$_k$).
	\end{thm}
	
	 Note that in Theorem \ref{main thm 1} $\out(\pi_1(M))$ is a finite group by Rips and Sela \cite{RS} building on ideas of Paulin. Moreover, the conditions ``$M$ is closed" and `` $k>0$ divisible by the order $|\out(\pi_1(M))|$" are necessary, see Example \ref{eg. 3-mfd with boudary} and Example \ref{eg. closed 3-mfd}. In fact, the above Theorem \ref{main thm 1} is a direct corollary of the following result.
	
	\begin{thm}\label{main thm 2}
		Let $f: M\times N\to M\times N$ be a homotopy equivalence, where $M$ is a closed negatively curved Riemannian manifold of odd dimension, and $N$ is a closed nilmanifold. Then for any integer $k>0$ divisible by the order $|\out(\pi_1(M))|$, we have $$\Lef(f^k)=\Min(f^k)=\Nie(f^k)=0,$$
		that is, the iterate $f^k$ can be fixed point free after a homotopy of $f^k$.
	\end{thm}

	
	Note in the above Theorem \ref{main thm 2}, $\Lef(f), \Min(f)$ and $\Nie(f)$ can be non-zero, see Example \ref{eg. closed 3-mfd}. More generally, we consider the product of an aspherical polyhedron with a nilmanifold, and show the following.
	
	\begin{thm}\label{main thm 3}
		Let $f: M\times N\to M\times N$ be a homotopy equivalence, where $M$ is a compact aspherical polyhedron and $N$ is a closed nilmanifold.  If $\pi_1(M)$ is centerless and $\out(\pi_1(M))$ is finite, then for any integer $k>0$ divisible by the order $|\out(\pi_1(M))|$, one of the following holds:
		\begin{enumerate}
			\item $\Lef(f^k)=\Min(f^k)=\Nie(f^k)=0,$ or
			\item there exists a lifting $\widetilde{f^k}$ of $f^k$ in a finite cover $\widetilde{M}\times N$ of $M\times N$, and a homeomorphism $h: \widetilde{M}\times N\to \widetilde{M}\times N$, such that $$h\comp \widetilde{f^k}\comp h^{-1}\simeq f_1\times f_2,$$
			where $f_1:\widetilde{M}\to \widetilde{M}$ and $f_2: N\to N$ are homotopy equivalences.
		\end{enumerate}
	\end{thm}
	
	The relation among the fixed point indices of a selfmap and its iterates is an active topic, see \cite{GT25, GJ09}, because it is closely related to the study of dynamical system. For some selfmaps on our product manifold $M\times N$, we have
	
	\begin{thm}\label{iterate Lef infy}
		Let $f: M\times N\to M\times N$ be a homotopy equivalence, where $M$ is a compact aspherical polyhedron and $N$ is a closed nilmanifold.  If $\chi(M)\ne 0$, $\pi_1(M)$ is centerless and the order $|\out(\pi_1(M))|=m$ is finite, then one of the following holds:
		\begin{enumerate}
			\item there exists an integer $\ell>0$ such that for every integer $k>0$ and $\ell|k$, we have $$\Lef(f^k)=\Min(f^k)=\Nie(f^k)=0;$$
			\item there exists an integer $\mathcal{B}_f>0$ such that for every integer $k>0$, we have $|\ind(f^{mk},\F)|\leq\mathcal{B}_f$ for every fixed point class $\F$ of $f^{mk}$. Moreover, $$|\Lef(f^{mk})|\to +\infty \quad and ~\quad\Nie(f^{mk})\to +\infty$$ when $k\to +\infty$.
		\end{enumerate}
	\end{thm}
	
	The paper is organized as follows. In Section \ref{sect 2}, we present some group-theoretical results utilized in this paper. In Section \ref{sect 3}, we introduce several useful lemmas to facilitate our proofs of the main theorems. In Section \ref{sect 4}--\ref{sect 6}, we prove the main theorems. Finally, in Section \ref{sect 7}, we provide some examples.

	\section{Preliminaries of group-theoretical results}\label{sect 2}

	In this section,  let $\Gamma$ be a centerless group (i.e. its center $C(\Gamma)$ is trivial), and let $G$ be a finitely generated nilpotent group. Let
	$$\varphi: \Gamma\times G \rightarrow \Gamma\times G$$
	be an automorphism. In \cite[Lemma 2.4]{N20} Neofytidis showed the following:
	
	\begin{lem}[\cite{N20}]\label{auto for product}
		Let $\Gamma$ be a centerless group (i.e. its center $C(\Gamma)$ is trivial), and let $G$ be a finitely generated nilpotent group. Then any automorphism $\varphi: \Gamma\times G \rightarrow \Gamma\times G$ has a form
		$$\varphi(\gamma,g)=(\alpha(\gamma),\rho(\gamma)L(g)),~(\gamma,g)\in \Gamma\times G,$$
		where $\alpha : \Gamma\rightarrow \Gamma$ and $L : G \rightarrow G$ are automorphisms and $\rho :  \Gamma\rightarrow C(G)$ is a homomorphism into the center of $G$.
	\end{lem}
	
	Now, we suppose that the fixed subgroup of $L|_{C(G)}:C(G)\to C(G)$ is trivial, that is,
	$$\fix(L|_{C(G)}):=\{ g\in C(G)\mid L(g)=g\}=1.$$
	Consider the endomorphism $\psi: C(G)\to C(G)$ defined by
	\begin{equation}\label{eq. psi}
		\psi(g)=L(g)g^{-1}.
	\end{equation}
	Note that
	$\ker\psi=\fix(L|_{C(G)})=1,$ so $\psi$ is injective and hence $\psi^{-1}:\psi(C(G))\to C(G)$ is an isomorphism. Let
	\begin{equation}\label{eq. Normal subgp}
		\Gamma':=\rho^{-1}(\psi(C(G)))\leq \Gamma.
	\end{equation}
	Then $\Gamma'=\Gamma$ if $\psi$ is an automorphism of $C(G)$.  Moreover, we have the following observation.
	\begin{lem}\label{Gamma' finite index in Gamma}
		Let $\psi$ defined in Eq. (\ref{eq. psi}) be an injective homomorphism. Then $\Gamma'$ is a finite-index normal subgroup of $\Gamma$.
	\end{lem}
	\begin{proof}
		Recall that any subgroup of a finitely generated nilpotent group is still finitely generated. So the center $C(G)$ of the finitely generated nilpotent group $G$ is finitely generated and abelian. Moreover, since the endomorphism $\psi: C(G)\to C(G)$ is injective, so the index $[C(G): \psi(C(G))]$ is finite. It follows that $$[\Gamma:\Gamma']\leq [C(G): \psi(C(G))]<+\infty.$$ The normality of $\Gamma'$ is apparent.
	\end{proof}
	
	Now we suppose that the automorphism $\a\in \aut(\Gamma)$ is an inner automorphism of $\Gamma$. Then we can define a homomorphism
	$$\theta=\psi^{-1}\comp \rho : ~\Gamma' \to  C(G)$$
	and an automorphism $h_\#: \Gamma'\times G\to \Gamma'\times G$ as follows,
	\begin{equation}\label{equ. h}
		h_\#(\gamma, g)=(\gamma, ~\theta(\gamma)g)
	\end{equation}
	with $h_\#^{-1}(\gamma, g)=(\gamma, \theta(\gamma)^{-1}g)$.
	Moreover, for any $(\gamma, g)\in \Gamma'\times G$, we have
	\begin{eqnarray}
		h_\#\comp\varphi\comp h_\#^{-1}(\gamma, g)
		&=& h_\#\big(\alpha(\gamma), ~\rho(\gamma)L(\theta(\gamma)^{-1}g)\big)\notag\\
		&=& \big(\alpha(\gamma), ~\theta(\alpha(\gamma)) \rho(\gamma)L(\theta(\gamma)^{-1})L(g)\big)\notag\\
		&=& \big(\alpha(\gamma), ~L(\theta(\gamma)^{-1})\theta(\gamma)\cdot \rho(\gamma)L(g)\big)\notag\\
		&=& \big(\alpha(\gamma), ~\psi(\theta(\gamma)^{-1})\cdot \rho(\gamma)L(g)\big)\notag\\
		&=& \big(\alpha(\gamma), ~\rho(\gamma^{-1})\cdot \rho(\gamma)L(g)\big)\notag\\
		&=& \big(\alpha(\gamma), ~L(g)\big)\in  \Gamma'\times G.\notag
	\end{eqnarray}
	In conclusion, we have proven:
	
	\begin{lem}\label{auto for product on finite index normal subgroup}
		Using the same notations as above. If the automorphism $\a\in \aut(\Gamma)$ is an inner automorphism of $\Gamma$ and $L\in\aut(G)$ satisfies $\fix(L|_{C(G)})=1$, then
		$$h_\# \comp \varphi \comp h_\#^{-1}=\alpha\times L: ~\Gamma'\times G\to \Gamma'\times G,$$
		namely,  ~$h_\#\comp\varphi\comp h_\#^{-1}(\gamma, g)=(\alpha(\gamma), L(g))$ for any $(\gamma, g)\in \Gamma'\times G$. In particular, $\Gamma'=\Gamma$ if the homomorphism $\psi$ in Eq. (\ref{eq. psi}) is an automorphism of $C(G)$.
	\end{lem}
	
	\section{Several useful lemmas}\label{sect 3}
	
	In this section, let $M$ be a compact connected aspherical polyhedron with centerless fundamental group $\pi_1(M)=\Gamma$, and let $N=\widehat{N}/G$ be a closed nilmanifold with fundamental group $\pi_1(N)=G$, where $\widehat{N}$ is a simply connected nilpotent Lie group with its center $C(\widehat{N})$ a Euclidean space and $G \subset \widehat{N}$ a cocompact lattice.
	
	Note that all the spaces $M, N$ and $M\times N$ are aspherical and $C(G)=G\cap C(\widehat{N})$. Moreover,  the product $M\times N$ satisfies the following mild condition: it is a compact connected polyhedron without local separating points, but it is not a surface (closed or with boundary) of negative Euler characteristic. Then by Jiang's deep result \cite[Main Theorem]{J80} or \cite[p.22, Theorem 6.3]{J83}, we have the following.
	
	\begin{lem}\label{Jiang's deep thm Mf=Nf}
		$\Min(f)=\Nie(f)$ for any selfmap $f$ of $M\times N$.
	\end{lem}
	
	Now, let $f : M \times N \rightarrow M \times N$ be a homotopy equivalence. Then by the above Lemma \ref{auto for product}, $f$ induces an automorphism $f_\#$ on the fundamental group:
	\begin{eqnarray}\label{eq. aut on fundamental group}
		f_\#: \Gamma\times G &\to& \Gamma\times G\\
		(\gamma, g)&\mapsto& (\a(\gamma), \rho(\gamma)L(g)),\notag
	\end{eqnarray}
	where $\alpha : \Gamma\rightarrow \Gamma$ and $L : G \rightarrow G$ are automorphisms and $\rho :  \Gamma\rightarrow C(G)$ is a homomorphism into the center of $G$.
	Since all of $M, N$ and $M\times N$ are aspherical, $f$ can be homotopied to a
	fiber-preserving map (still denoted by $f$) as follows:
	\begin{equation}\label{fibration map}
		f(x,y) = (f_1 (x),f_2 (x,y)),
	\end{equation}
	where $f_1 : M \rightarrow M$ induces $f_{1\#} = \alpha\in \aut(\Gamma)$, and for any $x \in M$,
	$$f_{2,x}(\cdot):=f_2(x,\cdot) : N \rightarrow N$$
	induces $f_{2,x \#} = L\in \aut(G).$
	
	Note that every selfmap of a nilmanifold has a linearization by tori decomposition, which can be represented by a matrix. Let $\mathcal{L}$ be the linearization of $f_{2,x}:N\to N$.
	Then the matrix $\mathcal{L}$ can also be produced as the derivative $\mathcal{L}': \mathfrak{G}\to \mathfrak{G}$ on the Lie algebra $\mathfrak{G}$ of $G$ of the automorphism
	$L: G \to G$.  Although the matrices $\mathcal{L}$ and $\mathcal{L}'$ may in fact be slightly different, the Nielsen calculations
	involving $\det(\mathcal{L}-I)=\det(\mathcal{L}'-I)$ will not be changed. (See \cite[p.91]{K05}). Therefore, to simplify the notation,  we will still use $\mathcal{L}$ to denote $\mathcal{L}'$ in this paper.
	

	\begin{lem}\label{periodic point}
		Let $f : M \times N \rightarrow M \times N$ be a homotopy equivalence, and let $\mathcal{L}$ be defined as above. If the determinant $\det(\mathcal{L}^k-I)= 0$ for some integer $k>0$, then $$\Lef(f^k)=\Min(f^k)=\Nie(f^k)=0,$$
		that is, the iterate $f^k$ can be fixed point free after a homotopy of $f^k$.
	\end{lem}
	
	\begin{proof}
		By Eq. (\ref{fibration map}), after a homotopy, we can assume that the homotopy equivalence $f^k:  M \times N \rightarrow M \times N$ is a fiber-preserving map as follows:
		$$f^k(x,y)=(f^k_1(x),f'_{2,x}(y)),$$
		where the map $f'_{2,x}: N\to N$ induces $(f'_{2,x})_\#=L^k: G\to G$ for every $x\in M$.  Recall that the nilmanifold $N$ is a Jiang space, the Nielsen number
		$$\Nie(f'_{2,x})=|\Lef(f'_{2,x})|=|\det(\mathcal{L}^k-I)|= 0.$$
		Hence, by the ``Product formula for the Nielsen number'' \cite[p.88, Corollary 4.3]{J83}, we have
		$$\Nie(f^k)=\Nie(f^k_1)\cdot \Nie(f'_{2,x})= 0.$$
		Therefore, by Lemma \ref{Jiang's deep thm Mf=Nf}, we have
		$$\Lef(f^k)=\Min(f^k)=\Nie(f^k)=0,$$
		that is, the iterate $f^k$ can be fixed point free after a homotopy of $f^k$.
	\end{proof}
	
	As in Eq. (\ref{eq. Normal subgp}), let
	$$\Gamma':=\rho^{-1}(\psi(C(G)))\lhd \Gamma= \pi_1(M)$$
	and let $\widetilde{M}$ be the finite cover of $M$ with $\pi_1(\widetilde{M})=\Gamma'.$ Denote by $\widehat{M}$ the universal cover of $M$. The fundamental groups $\Gamma $ and $\Gamma'$ act on $\widehat{M}$ cocompactly by deck transformations, and we identify $\Gamma$ with an orbit of a base-point in $\widehat{M}$. Recall that $G=\pi_1(N)$. For the automorphism $h_\#: \Gamma'\times G\to \Gamma'\times G$ given in Eq. (\ref{equ. h}), we can show:
	
	\begin{lem}\label{h realise a homeomorphism}
		There exists a homeomorphism $h: \widetilde{M}\times N\to   \widetilde{M}\times N$ inducing the automorphism $h_\#(\gamma, g)=(\gamma, ~\theta(\gamma)g)$ given in Eq. (\ref{equ. h}).
	\end{lem}
	
	Although Lemma \ref{h realise a homeomorphism} can be directly derived from \cite[Lemma 6.3 \& Lemma 6.6]{GL16}, for the sake of readability and completeness, we provide the detailed proof as follows.
	
	\begin{proof}
		Triangulate the universal cover $\widehat{M}$ in an equivariant way so that $\Gamma'\subset \widehat{M}$ belongs to the $0$-skeleton. Extend  the homomorphism
		$$\theta=\psi^{-1}\comp \rho: ~\Gamma'\rightarrow C(G)\subset \widehat N$$
		to the $0$-skeleton of  $\widehat{M}$ equivariantly in an arbitrary way. Recall that the universal cover $\widehat{N}$ is connected, simply connected and aspherical. Therefore, we can extend $\theta$ equivariantly to all skeleta of $\widehat{M}$ by induction on dimension, and we approximate the resulting map by an equivariant map (still denoted by $\theta$) $\theta: \widehat M\to \widehat{N}$, that is,
		$$\theta(\gamma \hat x)=\theta(\gamma)\theta(\hat x), ~~\forall \gamma\in \Gamma', ~\hat x\in\widehat M.$$
		It is easy to see (by using charts) that this can be done without changing the values on $\Gamma' \subset \widehat{M}$.  Note that $\theta(\gamma)\in C(G)$ for every $\gamma\in \Gamma'$. Therefore, the above equivariant map $\theta: \widehat M\to \widehat{N}$ descends to a map
		$$\theta: \widetilde{M} =\widehat{M}/\Gamma'\longrightarrow N=\widehat{N}/G.$$ Then the posited homeomorphism $h: \widetilde{M}\times N\to   \widetilde{M}\times N$ can be defined as follows:
		$$h(x,y)=(x,\theta(x)y).$$
	\end{proof}

	\begin{lem}\label{alpha is inner}
		Let $f : M \times N \rightarrow M \times N$ be a homotopy equivalence with a fixed point, and let $\a$ and $\mathcal{L}$ be defined as above. If $\a$ is an inner automorphism of $\Gamma$ and the determinant $\det(\mathcal{L}-I)\neq 0$, then there exists a finite cover $\widetilde{M}\times N$ of $M\times N$, and exists a lifting $\tilde{f}$ of $f$ in $\widetilde{M}\times N$ such that the following diagram is homotopy commutative:
		$$
		\xymatrix{
			\widetilde{M}\times N \ar[r]^{\tilde{f}}\ar[d]^{\cong}_{h} & \widetilde{M}\times N \ar[d]^{h}_{\cong}\\
			\widetilde{M}\times N \ar[r]^{f_1\times f_2} & \widetilde{M}\times N
		}
		$$
		
		where $h$ is a homeomorphism, and $f_1:\widetilde{M}\to \widetilde{M}$,  $f_2: N\to N$ are homotopy equivalences. In particular, if $|\det(\mathcal{L}-I)|=1$, then $\widetilde{M}=M$, and $f=\tilde f$ with
		$$\Lef(f)=\Lef(f_1)\cdot \Lef(f_2), \quad \Nie(f)=\Nie(f_1)\cdot \Nie(f_2).$$
	\end{lem}
	
	\begin{proof}
		Suppose $(x,y)\in M\times N$ is a fixed point of $f$, $\Gamma=\pi_1(M, x)$ and $G=\pi_1(N, y)$ throughout the proof. Then
		the automorphism $f_\#$ of $\Gamma\times G$ induced by $f$ has the following form as in Eq. (\ref{eq. aut on fundamental group}):
		$$f_\#(\gamma, g)=(\alpha(\gamma),\rho(\gamma)L(g)),~\mathrm{for}~(\gamma,g)\in \Gamma\times G,$$
		where $\alpha : \Gamma\rightarrow \Gamma$ and $L : G \rightarrow G$ are automorphisms and $\rho :  \Gamma\rightarrow C(G)$ is a homomorphism into the center of $G$. Since the determinant $\det(\mathcal{L}-I)\neq 0$, the fixed subgroup $\fix(L|_{C(G)})=1$, by Lemma \ref{auto for product on finite index normal subgroup}, there exists a finite-index normal subgroup $\Gamma'$ of $\Gamma$, and an automorphism $h_\#: \Gamma'\times G\to \Gamma'\times G$ such that
		\begin{equation}\label{conjugate product on pi}
			h_\# \comp f_\# \comp h_\#^{-1}=\alpha\times L: ~\Gamma'\times G\to \Gamma'\times G.
		\end{equation}
		
		Let $\widetilde{M}$ be the finite cover of $M$ corresponding to $\Gamma'=\pi_1(\widetilde{M}, \tilde x)$ for $\tilde x\in p^{-1}(x)$, and let $p\times\mathrm{id}:\widetilde{M}\times N\rightarrow M\times N$ be the covering map corresponding to the subgroup $\Gamma'\times G$ of $\Gamma\times G$. Recall that $\Gamma'\leq \Gamma $ is normal and $\a\in \inn(\Gamma)$, so $f_\#(\Gamma'\times G)\leq \Gamma'\times G$. Therefore, we have a commutative diagram
		$$
		\xymatrix{
			\widetilde{M}\times N \ar[r]^{\tilde{f}}\ar[d]_{p\times \mathrm{id}} & \widetilde{M}\times N \ar[d]^{p\times \mathrm{id}}\\
			M\times N \ar[r]^{f} & M\times N
		}
		$$
		where $\tilde{f}$ is a lifting of $f$.  Since all the spaces above are aspherical, there are homotopy equivalences
		\begin{eqnarray}
			f_1: (\widetilde{M}, \tilde{x}) &\to& (\widetilde{M}, \tilde{x}),\notag\\
			f_2: (N, y)&\to& (N, y), \notag\\
			h: (\widetilde{M}\times N, (\tilde{x},y))&\to& (\widetilde{M}\times N, (\tilde{x},y))\notag
		\end{eqnarray}
		inducing $f_{1\#}=\a|_{\Gamma'}, ~f_{2\#}=L$ and $h_\#$ respectively. Note that $h$  can be realized as a homeomorphism by Lemma \ref{h realise a homeomorphism}. Then on the fundamental group $\Gamma'\times G$, we have
		\begin{eqnarray}\nonumber
			(h\comp \tilde{f})_\# = h_\# \comp f_\# =(\alpha\times L)\comp h_\#=((f_1\times f_2)\comp h)_\#,
		\end{eqnarray}
		where the second ``=" holds by Eq. (\ref{conjugate product on pi}).
		Therefore, $h\comp \tilde{f}$ is homotopic to $(f_1\times f_2)\comp h$.
		
		In particular, if $|\det(\mathcal{L}-I)|=1$, then $(\mathcal{L}-I)^{-1}$ is an integer matrix and hence the homomorphism $\psi$ in Eq. (\ref{eq. psi}) is an automorphism. It follows $\Gamma'=\Gamma$ from Lemma \ref{auto for product on finite index normal subgroup}, and hence $\widetilde{M}=M$ and $\tilde f=f$. Therefore, $f$  is conjugate to $f_1\times f_2$ , by the ``Homotopy type invariance of the Nielsen number" \cite[p.21, Theorem 5.4]{J83}, ``Product formula for the Lefschetz number" \cite[p.85, Theorem 3.2]{J83} and ``Product formula for the Nielsen number" \cite[p.88, Theorem 4.1]{J83}, we have $\Lef(f)=\Lef(f_1)\cdot \Lef(f_2)$ and $\Nie(f)=\Nie(f_1)\cdot \Nie(f_2).$
	\end{proof}
	
	
	\section{Proof of Theorem \ref{main thm 0}}\label{sect 4}
	
	Recall that $\Sigma_g$ denotes the closed orientable surface of genus $g>1$. The fundamental group $\pi_1(\Sigma_g)$ has a canonical presentation:
	$$
	\pi_1(\Sigma_g)=\langle a_1,b_1,\ldots,a_g,b_g\mid [a_1,b_1]\cdots[a_g,b_g]=1 \rangle,
	$$
	where $[a, b]=aba^{-1}b^{-1}$.
	First, we have the following product formula for the index of a fixed point class.
	
	\begin{lem}\label{product formula for the index}
		Let $p: E=M\times N\to M$ be the projection to the first factor, where $\chi(M)\neq 0$ and $N$ is a closed nilmanifold. Let $f:E\to E$ be a fiber-preserving map inducing the identity $\id: M\to M$ on the base space, i.e., the following diagram is commutative.
		$$
		\xymatrix{
			E \ar[r]^f\ar[d]_p & E \ar[d]^p\\
			M\ar[r]^{\id} & M
		}
		$$
		Then, for any essential fixed point class $\F$ of $f$, the projection $p(\F)=M$ and
		$$|\ind(f,\F)|=[\pi_1(M) : p_{\#}(\fix f_{\#})]\cdot|\chi(M)|,$$
		where $f_{\#}: \pi_1(E, e)\to\pi_1(E,e)$ is the natural homomorphism induced by $f$  for $e\in \F$.
	\end{lem}
	
	\begin{proof}
		Note that for any selfmap of the closed nilmanifold $M$, its essential fixed point classes have the same index $\pm 1$ \cite{K05}. Then the conclusion follows from \cite[p.78, Theorem 1.6]{J83} and the ``Product formula for the index of a fixed point class" \cite[p.85, Theorem 3.3]{J83}.
	\end{proof}
	
	Now, we show the following lemmas for convenience of our discussion.
	
	\begin{lem}\label{invert}
		Let $L_m=
		\begin{pmatrix}
			m+1 & m \\
			1 & 1  \\
		\end{pmatrix}
		$ be a matrix and let $I$ be the identity matrix. Then the matrix $I-L_m^k$ is invertible for all  integers $k,m>0$.
	\end{lem}
	
	\begin{proof}
		By direct calculation, we obtain the eigenvalues of $L_m$ are $\lambda_{1,2}=\frac{m+2\pm \sqrt{m^2+4m}}{2},$ which are not unit roots for every integer $m>0$. So the eigenvalues of $L^k$ are not equal to $1$ and hence  $I-L_m^k$ is invertible for all  integers $k,m>0$.
	\end{proof}
	
	\begin{lem}\label{lem index m}
		Let $p: \pi_1(\Sigma_g)\times \Z\to \pi_1(\Sigma_g)$ be the projection to the first factor, and let
		$\rho: \pi_1(\Sigma_g)\to \Z$ be an epimorphism.	
		For any integer $m>0$, let
		$$H_m=\{(u,s)\in\pi_1(\Sigma_g)\times \Z \mid s=-\frac{\rho(u)}{m}, ~\rho(u)\equiv0\mod m\}.$$
		Then $p(H_m)$ is a subgroup of $\pi_1(\Sigma_g)$ of index
		$[\pi_1(\Sigma_g):p(H_m)]=m.$
	\end{lem}
	
	\begin{proof}
		Note that $\rho: \pi_1(\Sigma_g)\to \Z$ is an epimorphism, and $$p(H_m)=\{u\in\pi_1(\Sigma_g)\mid \rho(u)\equiv0\mod m\},$$
		so
		$[\pi_1(\Sigma_g): p(H_m)]=[\Z:m\Z]=m.$
	\end{proof}
	
	\begin{proof}[\textbf{Proof of Theorem \ref{main thm 0}}]
		By Theorem \ref{ZZ cout-examp.}, $\Sigma_g\times S^1$ has BIPH and hence it has BIPH$_k$ for any integer $k>0$ clearly. Now, let
		$\rho: \pi_1(\Sigma_g)\to \pi_1(S^1)=\Z$ be an epimorphism defined as follows:
		\begin{equation}\label{eq. rho to torus}\nonumber
			\rho(a_1)=1, ~\rho(b_1)=\rho(a_2)=\rho(b_2)=\cdots=\rho(a_g)=\rho(b_g)=0.  \end{equation}
		
		\textbf{(1).} Let $G=\pi_1(\Sigma_g)\times \pi_1(S^1)$. For any integer $m>0$, we pick an endomorphism $\phi\in\mathrm{End}(G)$ as
		$$\phi(u,s)=(u, \rho(u)+(m+1)s),~(u,s)\in \pi_1(\Sigma_g)\times \pi_1(S^1).$$
		By direct calculation, we obtain
		$$\phi^k(u,s)=(u, \sum^{k-1}_{i=0}(m+1)^i\rho(u)+(m+1)^ks).$$
		Hence
		\begin{eqnarray}
			\fix(\phi^k)
			& = & \{(u,s)\mid s=\sum^{k-1}_{i=0}(m+1)^i\rho(u)+(m+1)^ks\}\notag\\
			& = & \{(u,s)\mid (1-(m+1))^k)s=\sum^{k-1}_{i=0}(m+1)^i\rho(u)\}\notag\\
			& = & \{(u,s)\mid s=-\frac{\rho(u)}{m}, ~\rho(u)\equiv0\mod m\}.\notag
		\end{eqnarray}
		By Lemma \ref{lem index m}, we have $[\pi_1(\Sigma_g): p(\fix (\phi^k))]=m.$
		
		Moreover, note that $S^1$ is aspherical, we can pick a fiber-preserving map
		\begin{eqnarray}
			f: \Sigma_g\times S^1&\to& \Sigma_g\times S^1\notag\\
			(x, ~y)&\mapsto& (x, ~f_2(x, y))\notag
		\end{eqnarray}
		such that the induced homomorphism $f_{\#}=\phi$ on the fundamental group. The restriction of $f^k$ on each fiber $S^1$ is a selfmap of degree $(m+1)^k$, and hence has Lefschetz number $1-(m+1)^k$. By the ``Product formula for the Lefschetz number" \cite[p.85, Theorem 3.2]{J83}, we have
		\begin{equation}\label{eq-Lef1}\nonumber
			\Lef(f^k) = (1-(m+1)^k)\cdot \Lef(\id) = (1-(m+1)^k)\cdot \chi(\Sigma_g)\neq 0.
		\end{equation}
		Therefore, there is an essential fixed point class $\F$ of $f^k$ and by Lemma \ref{product formula for the index}, we have
		$$|\ind(f^k,\F)|=[\pi_1(\Sigma_g): p(\fix (\phi^k))]\cdot |\chi(\Sigma_g)|=m|\chi(\Sigma_g)|=m(2g-2).$$
		Recall that $m$ can be arbitrarily large,  we have proven that $\Sigma_g\times S^1$ does not have BIP$_k$ for any integer $k\geq 1$.\\
		
		\textbf{(2).} We first consider the case $n=2$. Let $G=\pi_1(\Sigma_g)\times \pi_1(T^2)$. For any integer $m>0$, let the matrix $L_m=
		\begin{pmatrix}
			m+1 & m \\
			1 & 1  \\
		\end{pmatrix}$  as in Lemma \ref{invert}, which can be viewed as an automorphism of $\pi_1(T^2)=\Z\times \Z$.  Now, we pick  an automorphism $\phi\in\aut(G)$ as
		$$\phi(u,{{s_1}\choose{s_2}})=(u,{{\rho(u)}\choose{0}}+L_m{{s_1}\choose{s_2}}),$$
		where $u\in \pi_1(\Sigma_g)$ and $(s_1,s_2)^T\in \pi_1(T^2)= \Z\times \Z$. Then for any integer $k>0$,
		$$\phi^k(u,{{s_1}\choose{s_2}})=(u,\sum^{k-1}_{i=0}L_m^{i}{{\rho(u)}\choose{0}}+L_m^k{{s_1}\choose{s_2}}).$$
		Note that the matrix $I-L_m^k=(I-L_m)(\sum^{k-1}_{i=0}L_m^{i})$ is invertible by Lemma \ref{invert},  we have
		\begin{eqnarray}
			\fix(\phi^k)
			& = & \{(u,{{s_1}\choose{s_2}})\mid {{s_1}\choose{s_2}}=\sum^{k-1}_{i=0}L_m^{i}{{\rho(u)}\choose{0}}+L_m^k{{s_1}\choose{s_2}}\}\notag\\
			& = & \{(u,{{s_1}\choose{s_2}})\mid (I-L_m^k){{s_1}\choose{s_2}}=\sum^{k-1}_{i=0}L_m^{i}{{\rho(u)}\choose{0}}\}\notag\\
			& = & \{(u,{{s_1}\choose{s_2}})\mid (I-L_m){{s_1}\choose{s_2}}={{\rho(u)}\choose{0}}\}\notag\\
			& = & \{(u,{{s_1}\choose{s_2}})\mid {{s_1}\choose{s_2}}={{0}\choose{-\frac{\rho(u)}{m}}}, ~\rho(u)\equiv 0\mod m\},\notag
		\end{eqnarray}
		where the condition ``$\rho(u)\equiv 0\mod m$'' is necessary because $s_2\in \Z$.
		Then, by Lemma \ref{lem index m}, we have
		$$[\pi_1(\Sigma_g): p(\fix (\phi^k))]=m,$$
		where $p: \pi_1(\Sigma_g)\times \pi_1(T^2)\to \pi_1(\Sigma_g)$ is the projection to the first factor.
		
		Moreover, note that $T^2$ is aspherical, we can pick a fiber-preserving homeomorphism
		\begin{eqnarray}\label{eq. f on T2}
			f: \Sigma_g\times T^2&\to& \Sigma_g\times T^2\\
			(x, ~y)&\mapsto& (x, ~f_2(x, y))\notag
		\end{eqnarray}
		such that the induced automorphism $f_{\#}=\phi$ on the fundamental group. The restriction of $f^k$ on each fiber $T^2$ has Lefschetz number $\det(I-L^k_m)\ne0$ from Lemma \ref{invert}. By the ``Product formula for the Lefschetz number" \cite[p.85, Theorem 3.2]{J83}, we have
		\begin{equation}\label{eq-Lef1}\nonumber
			\Lef(f^k) = \det(I-L^k_m)\cdot \chi(\Sigma_g)\ne 0.
		\end{equation}
		Therefore, there is an essential fixed point class $\F$ of $f^k$ and by Lemma \ref{product formula for the index}, we have
		\begin{equation}\label{eq. ind fk}
			|\ind(f^k,\F)|=[\pi_1(\Sigma_g): p(\fix (\phi^k))]\cdot |\chi(\Sigma_g)|=m|\chi(\Sigma_g)|=m(2g-2).
		\end{equation}
		Recall that $m$ can be arbitrarily large,  we have proven that $\Sigma_g\times T^2$ does not have BIPH$_k$ for any integer $k\geq 1$.\\
		
		Finally, for the remaining case $\Sigma_g\times T^n(n>2)$,  let $G=\pi_1(\Sigma_g)\times \pi_1(T^n)$. For any integer $m>0$, let the $n\times n$ matrix
		$$
		L_m=\begin{pmatrix}
			m+1&m&m&\cdots&m&m\\
			1&1&0&\cdots&0&0\\
			1&1&1&\cdots&0&0\\
			\vdots&\vdots&\vdots&\ddots&\vdots&\vdots\\
			1&1&1&\cdots&1&0\\
			1&1&1&\cdots&1&1\\
		\end{pmatrix}.
		$$
		Note that $\det(\lambda I-L_m)=(\lambda-1)^n-m\lambda^{n-1}$, for any integer $k>0$ and sufficiently large $m$, no roots of unity are eigenvalues of $L_m$ and hence $\det(I-L_m^k)\neq 0$. Moreover, $L_m$ can be viewed  as an automorphism of $\pi_1(T^n)=\Z^n$.
		Let $\phi\in\aut(G)$ be defined as
		$$\phi(u,\begin{pmatrix}
			s_1\\
			s_2\\
			\vdots\\
			s_n\\
		\end{pmatrix})=(u,\begin{pmatrix}
			\rho(u)\\
			0\\
			\vdots\\
			0\\
		\end{pmatrix}+L_m\begin{pmatrix}
			s_1\\
			s_2\\
			\vdots\\
			s_n\\
		\end{pmatrix})$$
		where $u\in \pi_1(\Sigma_g)$ and $(s_1,s_2,\ldots,s_n)^{T}\in \pi_1(T^n)= \Z^n$.
		Now, we can pick a fiber-preserving homeomorphism
		\begin{eqnarray}\nonumber
			f: \Sigma_g\times T^n&\to& \Sigma_g\times T^n\\
			(x, ~y)&\mapsto& (x, ~f_2(x, y))\notag
		\end{eqnarray}
		such that $f_\#=\phi.$ Using the same method of the above case ``$n=2$", we can show there is also an essential fixed point class $\F$ of $f^k$ such that
		$$|\ind(f^k,\F)|=m|\chi(\Sigma_g)|=m(2g-2).$$
		Thereofore, $\Sigma_g\times T^n(g, n\geq 2)$ does not have BIPH$_k$ for any integer $k\geq 1$.
	\end{proof}
	
	\section{Proofs of Theorems \ref{main thm 1}-\ref{main thm 3}}\label{sect 5}
	
	\begin{proof}[\textbf{Proof of Theorem \ref{main thm 1}}]
		For every homotopy equivalence $f: M\times N\to M\times N$ and  any integer $k>0$ divisible by the order $|\out(\pi_1(M))|$,  by Theorem \ref{main thm 2}, we have $\Nie(f^k)=0$. Therefore, $\ind(f^k, \F)=0$ for every fixed point class $\F$ of $f^k$. Hence $M\times N$ has BIPHE$_k$.
	\end{proof}
	
	\begin{proof}[\textbf{Proof of Theorem \ref{main thm 2}}]
		Since $M$ is a closed negatively curved Riemannian manifold of odd dimension, $\chi(M)=0$, $\pi_1(M)$ is centerless and $\out(\pi_1(M))$ is a finite group by Rips and Sela \cite{RS} building on ideas of Paulin. Let the integer $k>0$ be divisible by the order $|\out(\pi_1(M))|$.
		If $f^{k}$ is fixed point free, it is clear that
		$$\Lef(f^k)=\Min(f^k)=\Nie(f^k)=0.$$
		
		Now, we suppose that $(s,t)\in M\times N$ is a fixed point of $f^k$, $\Gamma=\pi_1(M, s)$ and $G=\pi_1(N, t)$ throughout the proof. By Eq. (\ref{fibration map}), after a homotopy, we can assume that the homotopy equivalence $f^k:  M \times N \rightarrow M \times N$ is a fiber-preserving map as follows:
		$$f^k(x,y)=(f^k_1(x),f'_{2,x}(y)),$$
		where the map $f^k_1: M\to M$ induces $(f^k_1)_\#=\alpha^k\in\aut(\Gamma)$, $f'_{2,x}: N\to N$ induces $(f'_{2,x})_\#=L^k\in\aut(G)$ for every $x\in M$.
		Since $\a\in \aut(\Gamma)$ and $k>0$ is divisible by the order $|\out(\Gamma)|$, we have $\alpha^k\in \inn(\Gamma)$. Moreover, recall that a closed negatively curved Riemannian manifold is aspherical,  so $M$ is aspherical and hence the map $f^k_1: M\to M$ can be homotopied to the identity $\id_M$ of $M$. Since $\chi(M)=0$, the unique nonempty fixed point class of $\id_M$ is $M$ itself with index
		$$\ind(\id_M,  M)=\chi(M)=0.$$
		It  follows that $\Nie(\id_M)=0$ and hence
		$$\Nie(f^k_1)=\Lef(f_1^k)=\Lef(\id_M)=\Nie(\id_M)=0.$$
		Then by the ``Product formula for the Nielsen number'' \cite[p.88, Corollary 4.3]{J83}, we have
		$$\Nie(f^k)=\Nie(f_1^k)\cdot \Nie(f'_{2,x})=0.$$
		Therefore, by Lemma \ref{Jiang's deep thm Mf=Nf}, we have
		$$\Lef(f^k)=\Min(f^k)=\Nie(f^k)=0.$$
		We complete our proof.
	\end{proof}

	\begin{proof}[\textbf{Proof of Theorem \ref{main thm 3}}]
		If $\Nie(f^k)=0$, then, by Lemma \ref{Jiang's deep thm Mf=Nf}, we obtain that
		$$\Lef(f^k)=\Min(f^k)=\Nie(f^k)=0.$$
		So Conclusion (1) holds.
		
		If $\Nie(f^k)\neq 0,$ then $f^k$ has a fixed point $(x,y)=f^k(x,y)\in M\times N$. Suppose $\Gamma=\pi_1(M)$ and $G=\pi_1(N)$ throughout the proof. Then
		the automorphism $f_\#$ of $\Gamma\times G$ induced by $f$ has the following form as in Eq. (\ref{eq. aut on fundamental group}):
		$$f_\#(\gamma, g)=(\alpha(\gamma),\rho(\gamma)L(g)),~\mathrm{for}~(\gamma,g)\in \Gamma\times G,$$
		where $\alpha : \Gamma\rightarrow \Gamma$ and $L : G \rightarrow G$ are automorphisms and $\rho :  \Gamma\rightarrow C(G)$ is a homomorphism into the center of $G$. Recall that $k>0$ is an integer divisible by the order $|\Out(\Gamma)|$, hence $\alpha^k : \Gamma\rightarrow \Gamma$ is an inner automorphism and
		\begin{equation}\label{fk on on group}
			f^k_\#(\gamma,g)=(\alpha^k(\gamma),\rho'(\gamma)L^k(g)),
		\end{equation}
		where $L^k : G \rightarrow G$ is an automorphism and $\rho' :  \Gamma\rightarrow C(G)$ is a homomorphism. Note that $\det(I-\mathcal{L}^k)\neq 0$. Otherwise, if $\det(I-\mathcal{L}^k)=0$ we obtain $\Nie(f^k)=0$ by Lemma \ref{periodic point}, it contradicts our assumption. Then applying Lemma \ref{alpha is inner} to $f^k$,
		there exists a lifting $\widetilde{f^k}$ of $f^k$ in a finite cover $\widetilde{M}\times N$ of $M\times N$, and a homeomorphism $h: \widetilde{M}\times N\to \widetilde{M}\times N$, such that $$h\comp \widetilde{f^k}\comp h^{-1}\simeq f_1\times f_2,$$
		where $f_1:\widetilde{M}\to \widetilde{M}$ and $f_2: N\to N$ are homotopy equivalences.
		So Conclusion (2) holds.
	\end{proof}

	\section{Lefschetz number and fixed point index of iterates}\label{sect 6}
	
	In this section, we show some further results on the Lefschetz number and fixed point index of iterates.
	Suppose that $M$ is a compact aspherical polyhedron with centerless fundamental group $\pi_1(M)=\Gamma$ and $N$ is a closed nilmanifold with fundamental group $\pi_1(N)=G$. In Theorem \ref{main thm 2}, the order $|\out(\pi_1(M))|$ is finite and $\chi(M)=0$, then for any homotopy equivalence $f:M\times N\to M\times N$ we can obtain $\Lef(f^k)=\Min(f^k)=\Nie(f^k)=0$ for every integer $k>0$ divisible by the order $|\out(\pi_1(M))|$. Hence, in this section, we consider the case that $\chi(M)\ne 0$.
	
	For the sake of readability, we emphasize the following facts mentioned in Section \ref{sect 3}.  Let  $f:M\times N\to M\times N$ be a homotopy equivalence. Then we have
	$$f_\#(\gamma, g)=(\alpha(\gamma),\rho(\gamma)L(g)), ~(\gamma,g)\in \Gamma\times G,$$
	where $\alpha\in \aut(\Gamma)$, $L\in\aut(G)$ and $\rho:\Gamma\to C(G)$ is a homomorphism.  Therefore, $f$ can be homotopied to a
	fiber-preserving map (still denoted by $f$) as follows:
	\begin{equation}\label{Eq. 6.1 fibration map'}
		f(x,y) = (f_1 (x),f_2 (x,y)),
	\end{equation}
	where $f_1 : M \rightarrow M$ induces $f_{1\#} = \alpha\in \aut(\Gamma)$, and for any $x \in M$,
	$$f_{2,x}(\cdot):=f_2(x,\cdot) : N \rightarrow N$$
	induces $f_{2,x \#} = L\in \aut(G).$ Let $\mathcal{L}$ be the linearization of $f_{2,x}:N\to N$.
	Note that the matrix $\mathcal{L}$ can also be produced as the derivative $\mathcal{L}': \mathfrak{G}\to \mathfrak{G}$ on the Lie algebra $\mathfrak{G}$ of $G$ of the automorphism
	$L: G \to G$.
	
	With the above notations, we have the following proposition.
	
	\begin{prop}\label{inn infy}
		Let $M$ be a compact aspherical polyhedron and $N$ a closed nilmanifold. If $\chi(M)\ne 0$ and $\pi_1(M)=\Gamma$ is centerless, then for any homotopy equivalence $f:M\times N\to M\times N$ with $f_{1\#}\in \inn(\Gamma)$,  one of the following holds:
		\begin{enumerate}
			\item if $\mathcal{L}$ has an eigenvalue that is a root of unity of order $\ell$, then $$\Lef(f^k)=\Min(f^k)=\Nie(f^k)=0,$$
			for every positive integer $k$ such that $\ell|k$, or
			\item if $\mathcal{L}$ has no eigenvalues that are roots of unity, then there exists an integer $\mathcal{B}_f>0$ such that for every integer $k>0$ and every fixed point class $\F$ of $f^k$, we have $|\ind(f^k,\F)|\leq\mathcal{B}_f$. Moreover, $$|\Lef(f^k)|\to +\infty \quad and ~\quad\Nie(f^k)\to +\infty$$ when $k\to +\infty$.
		\end{enumerate}
	\end{prop}
	
	\begin{proof}
		(1) $\mathcal{L}$ has an eigenvalue that is a root of unity of order $\ell$. Then for every positive integer $k$ such that $\ell|k$, we have $\det(I-\mathcal{L}^k)=0$, and hence by Lemma \ref{periodic point},
		$$\Lef(f^k)=\Min(f^k)=\Nie(f^k)=0.$$
		
		(2) $\mathcal{L}$ has no eigenvalues that are roots of unity. Then $\det(I-\mathcal{L}^k)\ne 0$ for every integer $k>0.$	
		
		Since $f_{1\#}\in \inn(\Gamma)$, by Eq. (\ref{Eq. 6.1 fibration map'}), after a homotopy, we can assume that
		the homotopy equivalence $f:  M \times N \rightarrow M \times N$ is a fiber-preserving homotopy equivalence as follows:
		$$f(x,y)=(x,f_{2,x}(y)),$$
		where the map $f_{2,x}: N\to N$ induces $(f_{2,x})_\#=L: \pi_1(N)=G\to G$ for every $x\in M$, and
		$$f_\#(\gamma, g)=(\gamma, ~\rho(\gamma)L(g)), ~(\gamma,g)\in \Gamma\times G,$$
		where $L\in\aut(G)$ and $\rho:\Gamma\to C(G)$ is a homomorphism.
		Therefore,
		$$f^k(x,y)=(x,f'_{2,x}(y)),$$
		where $f'_{2,x}: N\to N$ induces $(f'_{2,x})_\#=L^k\in\aut(G)$ for every $x\in M$,  and
		$$f^k_{\#}(\gamma, g)=(\gamma, ~\rho(\gamma)L(\rho(\gamma))\cdots L^{k-1}(\rho(\gamma))L^k(g)).$$
		Then applying Lemma \ref{product formula for the index}, for every essential fixed point class $\F$ of $f^k$, we have
		\begin{equation}\label{eq. 6.2}
			|\ind(f^k,\F)|=[\Gamma:p_\#(\fix f^k_\#)]\cdot|\chi(M)|,
		\end{equation}
		where
		\begin{equation}\label{eq. 6.3}
			p_\#(\fix f^k_\#)=\{\gamma\in\Gamma\mid \exists g\in G,~\mathrm{s.t.}~g=\rho(\gamma)L(\rho(\gamma))\cdots L^{k-1}(\rho(\gamma))L^k(g)\}.
		\end{equation}
		
		Let us consider the endomorphism $\psi:C(G)\to C(G)$ defined in Equation (\ref{eq. psi}),
		$$\psi(g)=L(g)g^{-1},$$
		which is a monomorphism (since $\mathcal{L}$ does not have eigenvalue $1$) and hence gives  an isomorphism $\psi:C(G)\to \psi(C(G))$. Let
		$$\Gamma':=\rho^{-1}(\psi(C(G)))\leq \Gamma.$$
		For any $\gamma\in \Gamma'$, let $g:=\psi^{-1}(\rho(\gamma^{-1}))=(I-\mathcal{L})^{-1}(\rho(\gamma))\in C(G)$ (the operation in the center $C(G)$ is denoted by addition in the following). Then
		\begin{eqnarray}
			\rho(\gamma)L(\rho(\gamma))\cdots L^{k-1}(\rho(\gamma))L^k(g)
			& = & \sum^{k-1}_{i=0}\mathcal{L}^{i}({\rho(\gamma)})+\mathcal{L}^k (g)\notag\\
			& = & (I-\mathcal{L}^{k})(I-\mathcal{L})^{-1}(\rho(\gamma))+\mathcal{L}^k (g)\notag\\
			& = & (I-\mathcal{L}^{k})(g)+\mathcal{L}^k (g)\notag\\
			& = & g.\notag
		\end{eqnarray}
		Therefore, by Eq. (\ref{eq. 6.3}), we have $\Gamma'\leq p_\#(\fix f^k_\#),$ and hence  by Lemma \ref{Gamma' finite index in Gamma}, we obtain
		$$[\Gamma:p_\#(\fix f^k_\#)]\leq [\Gamma:\Gamma']< +\infty$$  for every integer $k>0$. Then by Eq. (\ref{eq. 6.2}), we have proven$$|\ind(f^k,\F)|\leq\mathcal{B}_f:=[\Gamma:\Gamma']\cdot|\chi(M)|<+\infty.$$
		Note that the bound $\mathcal{B}_f$ only depends on $f$ and $M$, and is independent of $k$.
		
		Furthermore, recall that $\mathcal{L}$ is an invertible matrix over the integers with no eigenvalues that are roots of unity, then $\mathcal{L}$ has an eigenvalue $\lambda$ with $|\lambda|>1$. Hence, by the ``Product formula for the Lefschetz number'' \cite[p.85, Theorem 3.2]{J83} and the assumption $\chi(M)\ne 0$, we obtain
		$$|\Lef(f^k)|=|\Lef(\id)\cdot \Lef(f'_{2,x})|=|\chi(M)\det(I-\mathcal{L}^k)|\to +\infty, ~\mathrm{when}~k\to+\infty.
		$$
		Therefore, by the famous Lefschetz-Hopf theorem and Eq. (\ref{eq. 6.2}),
		$\Nie(f^k)$ must tend to $+\infty$ when $k\to+\infty$.
	\end{proof}
	
	As a corollary of Proposition \ref{inn infy}, we have:
	
	\begin{thmbis}{iterate Lef infy}
		Let $f: M\times N\to M\times N$ be a homotopy equivalence, where $M$ is a compact aspherical polyhedron and $N$ is a closed nilmanifold.  If $\chi(M)\ne 0$, $\pi_1(M)$ is centerless and the order $|\out(\pi_1(M))|=m$ is finite, then one of the following holds:
		\begin{enumerate}
			\item there exists some integer $\ell>0$ such that for every integer $k>0$ and $\ell|k$, we have $$\Lef(f^k)=\Min(f^k)=\Nie(f^k)=0;$$
			\item there exists an integer $\mathcal{B}_f>0$ such that for every integer $k>0$, we have $|\ind(f^{mk},\F)|\leq\mathcal{B}_f$ for every fixed point class $\F$ of $f^{mk}$. Moreover, $$|\Lef(f^{mk})|\to +\infty \quad and ~\quad\Nie(f^{mk})\to +\infty$$ when $k\to +\infty$.
		\end{enumerate}
	\end{thmbis}
	
	\begin{proof}
		Let $\pi_1(M)=\Gamma$ and $\pi_1(N)=G$. By Eq. (\ref{Eq. 6.1 fibration map'}), after a homotopy, we can assume that
		the homotopy equivalence $f:  M \times N \rightarrow M \times N$ is a fiber-preserving homotopy equivalence as follows:
		$$f(x,y)=(f_1(x),f_2(x,y)),~(x,y)\in M\times N,$$
		where $f_1 : M \rightarrow M$ induces $f_{1\#} = \alpha\in \aut(\Gamma)$. By the assumption $|\out(\pi_1(M))|=m<+\infty$, we have
		$$f^m_{1\#}=\alpha^m\in \inn(\pi_1(M)).$$
		Then, by applying Proposition \ref{inn infy} to $f^m$, the conclusion holds.
	\end{proof}
	
	In Theorem \ref{iterate Lef infy}, if $N$ is the $n$-dimension torus, then the upper bound $\mathcal{B}_f$ can be revised
	to an achievable bound $\mathcal{A}_f\leq \mathcal{B}_f$, see the following.
	
	\begin{cor}
		Let $f: M\times T^n\to M\times T^n$ be a homotopy equivalence, where $M$ is a compact aspherical polyhedron and $T^n(n\geq 1)$ is the $n$-dimension torus.  If $\chi(M)\ne 0$, $\pi_1(M)$ is centerless and the order $|\out(\pi_1(M))|=m$ is finite, then one of the following holds:
		\begin{enumerate}
			\item there exists some integer $\ell>0$ such that for every integer $k>0$ and $\ell|k$ we have $$\Lef(f^k)=\Min(f^k)=\Nie(f^k)=0;$$
			\item there exists an integer $\mathcal{A}_f>0$ such that for every integer $k>0$, we have $|\ind(f^{mk},\F)|=\mathcal{A}_f$ for every essential fixed point class $\F$ of $f^{mk}$. Moreover, $$|\Lef(f^{mk})|\to +\infty \quad and ~\quad\Nie(f^{mk})\to +\infty$$ when $k\to +\infty$.
		\end{enumerate}
	\end{cor}
	
	\begin{proof}
		Since the proof is almost identical to that of Proposition \ref{inn infy} except that $f$ is replaced by $f^m$, we use the same notations as in Proposition \ref{inn infy}. Since $\pi_1(T^n)=\Z^n$ is abelian, the linearization $\mathcal{L}=L=f_{2,x\#}$. If $L$ has some eigenvalue that is a root of unity of order $\ell$, then Item (1) directly follows from Proposition \ref{inn infy}(1). Otherwise, $L$ has no eigenvalues that are roots of unity. Let $\pi_1(M)=\Gamma$. Since $|\out(\pi_1(M))|=m$ is finite, after a homotopy, we can assume that
		$$f_\#^{mk}(\gamma,g)=(\gamma, ~\sum^{k-1}_{i=0}L^{mi}({\rho(\gamma)})+L^{mk}(g)),$$
		and hence Eq. (\ref{eq. 6.3}) becomes
		\begin{eqnarray}
			p_\#(\fix f^{mk}_\#)&=&\{\gamma\in \Gamma\mid \exists g\in \Z^n,  \mathrm{s.t.} ~g=\sum^{k-1}_{i=0}L^{mi}({\rho(\gamma)})+L^{mk}(g)\}\nonumber\\
			&=&\{\gamma\in \Gamma\mid \exists g\in \Z^n,  \mathrm{s.t.} ~(I-L^m)(g)=\rho(\gamma)\}\nonumber.
		\end{eqnarray}
		Note that $[\Gamma:p_\#(\fix f^{mk}_\#)]\leq [\Gamma:\Gamma'] <+\infty$ as in the proof of Proposition \ref{inn infy}. Now, set $$\mathcal{A}_f:=[\Gamma:p_\#(\fix f^{mk}_\#)]\cdot|\chi(M)|,$$
		which is independent of $k$.
		Then by Eq. (\ref{eq. 6.2}), $|\ind(f^{mk},\F)|=\mathcal{A}_f<+\infty$ for every essential fixed point class $\F$ of $f^{mk}$. The remaining part of Item (2) directly follows from Proposition \ref{inn infy}(2) by replacing $f$ by $f^m$.
	\end{proof}

	\section{Some examples}\label{sect 7}

	In \cite{PZ}, for each pair of integers $n, \ell$ such that $n\geq 3$, $0\leq \ell < n$ and $\mathrm{gcd}(n,2-\ell) = 1$, Paoluzzi and Zimmermann constructed a compact orientable hyperbolic 3-manifold $M_{n,\ell}$ with totally geodesic boundary (a surface of genus $n-1$). The fundamental group $\pi_1(M_{n,\ell})=G_{n,\ell}$ has the following presentation
	$$G_{n,\ell}=\langle x_0,x_1,\ldots, x_{n-1}\mid \prod \limits^{n-1}_{i=0}x_{i(2-\ell)}x^{-1}_{i(2-\ell)+1}x^{-1}_{(i+1)(2-\ell)-1}=1 \rangle,$$
	where we take indices $\mathrm{mod}~n$. Note that $\out(\pi_1(M_{n,\ell}))$ is finite (which can be obtain by applying Mostow rigidity to the double of $M_{n,\ell}$), but $M_{n,\ell}$ is not closed. The following example shows that the product $M_{n,\ell}\times T^2$ does not necessarily have BIPH$_k$ (and hence not have BIP$_k$). It implies that the condition ``$M$ is closed" in Theorem \ref{main thm 1} is necessary.
	
	\begin{exam}\label{eg. 3-mfd with boudary} Set $n=3$ and $\ell=1$. The fundamental group of $M_{3,1}$ has the following presentation
		$$G_{3,1}=\langle x_0,x_1, x_2\mid x_0x_1^{-1}x_0^{-1}x_1x_2^{-1}x_1^{-1}x_2x_0^{-1}x_2^{-1}=1 \rangle,$$
		where $M_{3,1}$ has totally geodesic boundary $\Sigma_2$. Let $G=G_{3,1}\times \pi_1(T^2)$. For any integer $
		m>0$,  let the matrix $L_m=
		\begin{pmatrix}
			m+1 & m \\
			1 & 1  \\
		\end{pmatrix}$  as in Lemma \ref{invert}. We pick $\phi\in\aut(G)$ as
		$$\phi(u,{{s_1}\choose{s_2}})=(u,{{\rho(u)}\choose{0}}+L_m{{s_1}\choose{s_2}})$$
		where $u\in G_{3,1}$, $(s_1,s_2)^T\in \pi_1(T^2)\cong \Z\times \Z$ and $\rho: G_{3,1}\to \Z$ is defined as follows:
		$$\rho(x_0)=1, ~\rho(x_1)=0,~\rho(x_2)=-1.$$
		By direct calculation, we obtain
		$$\phi^k(u,{{s_1}\choose{s_2}})=(u,\sum^{k-1}_{i=0}L_m^{i}{{\rho(u)}\choose{0}}+L_m^k{{s_1}\choose{s_2}}).$$
		Recall that $(I-L_m^k)=(I-L_m)(\sum^{k-1}_{i=0}L_m^{i})$ is invertible by Lemma \ref{invert}, hence
		\begin{eqnarray}
			\fix(\phi^k)
			& = & \{(u,{{s_1}\choose{s_2}})\mid {{s_1}\choose{s_2}}=\sum^{k-1}_{i=0}L_m^{i}{{\rho(u)}\choose{0}}+L_m^k{{s_1}\choose{s_2}}\}\notag\\
			& = & \{(u,{{s_1}\choose{s_2}})\mid (I-L_m^k){{s_1}\choose{s_2}}=\sum^{k-1}_{i=0}L_m^{i}{{\rho(u)}\choose{0}}\}\notag\\
			& = & \{(u,{{s_1}\choose{s_2}})\mid (I-L_m){{s_1}\choose{s_2}}={{\rho(u)}\choose{0}}\}\notag\\
			& = & \{(u,{{s_1}\choose{s_2}})\mid {{s_1}\choose{s_2}}={{0}\choose{-\frac{\rho(u)}{m}}}, ~\rho(u)\equiv 0\mod m\}\notag.
		\end{eqnarray}
		Since $\rho:G_{3,1}\rightarrow \Z $ is an epimorphism,
		it is not hard to see that
		$$[G_{3,1}:p(\fix(\phi^k))]=m,$$
		where $p: G_{3,1}\times \pi_1(T^2)\to G_{3,1}$ is the projection to the first factor.
		
		Now, we can pick a homeomorphism
		\begin{eqnarray}\nonumber
			f: M_{3,1}\times T^2&\to& M_{3,1}\times T^2\\
			(x, ~y)&\mapsto& (x, ~f_2(x, y))\notag
		\end{eqnarray} such that $f$ induces $\phi$ on the fundamental group. The restriction of $f^k$ on each fiber $T^2$ has Lefschetz number $\det(I-L^k_m)\ne0$ from Lemma \ref{invert}. By the ``Product formula for the Lefschetz number" \cite[p.85, Theorem 3.2]{J83}, we have
		\begin{equation}\label{eq-Lef1}\nonumber
			\Lef(f^k) = \det(I-L^k_m)\cdot \chi(M_{3,1})\ne 0.
		\end{equation}
		Therefore, there is an essential fixed point class $\F$ of $f^k$ and
		$$
		|\ind(f^k,\F)|=[G_{3,1}: p(\fix (\phi^k))]\cdot |\chi(M_{3,1})|=m|\chi(M_{3,1})|.
		$$
		Recall that $m$ can be arbitrarily large,  we have proven that $M_{3,1}\times T^2$ does not have BIPH$_k$ for any integer $k\geq 1$.
	\end{exam}

	Finally, the following example shows that the condition ``$k>0$ divisible by $|\Out(\pi_1(M ))|$'' is necessary in Theorem \ref{main thm 1}, even if $M$ is a closed hyperbolic $3$-manifold.
	
	\begin{exam}\label{eg. closed 3-mfd}
		Let $M$ be a closed orientable hyperbolic $3$-manifold with a retraction
		$$r: M\to\Sigma_g,$$
		where $\Sigma_g\subset M (g>1)$ is a totally geodesic embedded closed orientable hyperbolic surface fixed pointwise by an orientation-reversing isometry $f_1: M\to M$ of order $2$. (Such a manifold $M$ does exist, see \cite[Theorem 1.9]{BHW11}).
		Then, for any integer $m>0$, there is a homeomorphism
		\begin{eqnarray}\nonumber
			f: M\times T^2 &\to& M\times T^2\\
			(x, ~y)&\mapsto& (f_1(x), ~f_2(r(x),y))\notag
		\end{eqnarray}
		where $f_2: \Sigma_g\times T^2 \to T^2$ is defined as in Eq. (\ref{eq. f on T2}). Note that $$f|_{\Sigma_g\times T^2}(x, y)=(x, ~f_2(x,y)),$$
		and the fixed point set
		$$\fix f=\fix(f|_{\Sigma_g\times T^2}:\Sigma_g\times T^2\to\Sigma_g\times T^2)$$
		which is a nonempty fixed point class $\F:=\fix f$ (see \cite[Sect. 4]{ZZ23}). Since the restriction of $f_1$ to the neighborhood of $\Sigma_g$ in $M$ is a reflection, by Eq. (\ref{eq. ind fk}), we have the index
		$$|\ind(f, \F)|=|\ind(f|_{\Sigma_g\times T^2}, \F)|=m(2g-2).$$
		Since $m$ can be arbitrarily large, we have proven that $M\times T^2$ does not have BIPH. But it has BIPH$_k$ for integer $k>0$ divisible by $|\Out(\pi_1(M ))|$ according to Theorem \ref{main thm 1}.
	\end{exam}

	\noindent\textbf{Acknowledgements.} The authors are very grateful to Jiming Ma and Shengkui Ye for their valuable communications.


\end{document}